\documentclass[12pt]{amsart}
\usepackage[top=1.2in, left=1.2in, right=1.2in, bottom=1.2in]{geometry}
\usepackage[utf8]{inputenc}
\usepackage{amsmath,amsfonts,amssymb,epsfig}
\usepackage{delarray}
\usepackage{graphicx}
\newtheorem{thm}{Theorem}[section]

\newtheorem{lem}[thm]{Lemma}

\theoremstyle{definition}
\newtheorem{define}[thm]{Definition}
\theoremstyle{remark}

\numberwithin{equation}{section}

\newcommand{\bbR}{\mathbb R}

\newcommand{\bbC}{\mathbb C}

\newcommand{\bbK}{\mathbb K}

\begin{document}

\title{Normalization of rationally integrable systems}

\author{Nguyen Tien Zung}
\address{Institut de Mathématiques de Toulouse and Torus Actions SAS}
\email{ntzung@torus.ai}

\date{This version : January 2025}
\subjclass{37G05,70K45,34C14,70GXX}

\keywords{Local normal form, torus action, rational 
integrability, commuting vector fields}%

\begin{abstract}
It's well-known that any analytic vector field near a singular point admits a normalization
à la Poincaré-Birkhoff, but this normalization is only formal in general, and the problem of
analytic (convergent) normalization is a difficult one. In \cite{Zung_Birkhoff2005,Zung_Poincare2002} we proposed a new approach to the normalization of 
vector fields, via their intrinsic associated torus actions: an analytic vector field is analytically normalizable near a singular point if and only if its associated torus action is analytic (and not just formal). We then showed that if a vector field is analytically integrable then its associated torus action is analytic, therefore the vector field is analytically normalizable \cite{Zung_Birkhoff2005,Zung_Poincare2002}. 
In this paper, we extend this analytic normalization result to the case of \emph{rationally integrable} systems, where the first integrals and commuting vector fields are not required to be analytic, but just rational (i.e., quotients of analytic functions or vector fields by analytic functions). For example, any vector field of the type $X = fY$, where $Y$ is an analytically diagonalizable vector field and $f$ is an analytic function such that $Y(f) = 0$, is rationally integrable
but not necessarily analytically integrable. 

\bigskip

\emph{Dedicated to Anatoly Timofeevich Fomenko on the occasion of his 80th birthday.} 
\end{abstract}
\maketitle


\section{Introduction}


In order to state precisely the main result of this paper, let us first briefly recall the theory of local normalization of analytic vector fields (see, e.g., \cite{Bruno_Book1989,Roussarie,SiegelMoser1971}) from the point of view of torus actions \cite{Zung_Birkhoff2005,Zung_Poincare2002}. 
Consider a local analytic or formal vector field $X$ on $(\mathbb{K}^n, 0)$, where
$\mathbb{K}$ is $\mathbb{R}$ or $\mathbb{C}$. Then $X$ may be viewed as a linear operator $f \mapsto X(f)$ on the space of formal
functions on $(\mathbb{K}^n, 0)$. Similarly to the Jordan decomposition of every finite-dimensional square matrix into the sum of its semisimple part and its nilpotent part, (by projective limit) $X$ also admits a unique decomposition
\begin{equation}
X = X^S + X^N
\end{equation}
into the sum of its semisimple part $X^S$ and its nilpotent part
$X^N$. Here $X^S$ and $X^N$ are formal vector fields, both of them commute with $X$ and with each other, and $X^S$ is diagonalizable 
in a (complexified if $\mathbb{K} = \mathbb{R}$) formal coordinate system $(x_1,\hdots, x_n)$: 
\begin{equation}
X^S = \sum_{i=1}^n \lambda_i  x_i \dfrac{\partial}{\partial x_i},
\end{equation}
i.e. $X^S(x_i) = \lambda_i x_i$, and $X^S(\prod_i x_i^{a_i}) =
(\sum_i a_i \lambda_i) \prod_i x_i^{a_i}$ for the other monomial functions in this coordinate system. If we write the Taylor expansion of $X$ in this coordinate system $(x_1,\hdots, x_n)$ as
\begin{equation}
X = X^{s} + X^{n} + X^{(2)} + X^{(3)} + \hdots,
\end{equation}
where $X^{s}$ and  $X^{n}$ are the semisimple part and the nilpotent part
of the linear vector field $X^{(1)} = X^{s} + X^{n}$ in its Jordan decomposition, and $X^{(k)}$ is the term of degree $k$ of $X$, then we have that $X^S = X^{s}$ and $X^N = X^{n} + X^{(2)} + X^{(3)} + \hdots$. The equation $[X^S,X^N] = 0$ means that  
$[X^{s},X^{(k)}] = 0$ for every $k \geq 2$, i.e., all the non-linear terms of $X$ in the coordinate system $(x_1,\hdots, x_n)$
are \textit{\textbf{resonant}} with respect to the semisimple part of $X$. 

The above coordinate system $(x_1,\hdots, x_n)$, in which every non-linear term of $X$ is resonant, is called a \textbf{\textit{Poincaré-Birkhoff normalization}} of $X$, and the expression of $X$ in such a coordinate system is called a 
\textbf{\textit{Poincaré-Birkhoff normal form}}. It is a classical result that 
a Poincaré-Birkhoff normalization
of a formal or local analytic vector field $X$ always exists, but is only formal in general, even when $X$ is analytic \cite{Bruno_Book1989,Roussarie,SiegelMoser1971}. 

The problem of existence of
local \textit{analytic} normalization of vector fields 
is a difficult one. There are two main approaches to this problem: the analytic
approach, which deals with small denominator phenomena, using analytic
estimates and fast convergence methods and Diophantine conditions, with important
results due to Poincaré, Siegel, Bruno and other people, see, e.g.,  \cite{Bruno_Book1989,Roussarie}. The second approach is more geometric, where one
tries to use symmetries and first integrals to arrive at the existence of a local
analytic normalization. There are several works in this second approach, including Bambusi--Cicogna--Gaeta--Marmo  
\cite{BaCiGaMa}, Bruno, Cicogna and Walcher 
\cite{BrWa,CiWa}, and Stolovitch \cite{Stolovitch-IHES,Stolovitch-Cartan}, which are somehow still quite analytical. 

In  \cite{Zung_Birkhoff2005,Zung_Poincare2002}  we developed a new geometric method of normalization, based on \textit{torus actions}.
Our starting point is the following observation: the semisimple part 
$X^S$ of $X$ can be written as
\begin{equation}
X^S = \sum_{i=1}^\tau \gamma_i Z_i,
\end{equation}
where $\gamma_1, \hdots, \gamma_\tau$ are incommensurable complex numbers,
$Z_1,\hdots,Z_\tau$ are (formal) linearly independent vector fields which generate a torus  $\mathbb{T}^\tau$-action (in the complexified space if $X$ is real). In a normalizing coordinate system, where $X^S$ is diagonalized, the vector fields $Z_i$ are also diagonal, with coefficients which are integer multiples of $\sqrt{-1}$. This torus action is intrinsic, i.e. uniquely determined by $X$ up to automorphisms of the torus, and is called the intrinsic \textbf{\textit{associated torus action}} of $X$, and the number $\tau$ is called the \textbf{\textit{toric degree}} of $X$ at $0$. The linearization of this intrinsic torus action is equivalent to the normalization of $X$, and so \textit{$X$ admits a local analytic normalization if and only if this intrinsic associated torus action is analytic}.
 We then used geometric approximation methods to show the analyticity of 
these associated torus actions, which leads to the existence of analytic normalization for analytically integrable vector fields \cite{Zung_Birkhoff2005,Zung_Poincare2002}, without any additional assumption. This is a significant improvement over
previous results of Russmann \cite{Russmann1964}, Vey \cite{Vey1978,Vey2}, Ito \cite{Ito_Birkhoff1989,Ito_Birkhoff1992}, Kappeler--Kodama--Nemethi \cite{KKN_Birkhoff1998}, among others.

The results of \cite{Zung_Birkhoff2005,Zung_Poincare2002}
are valid for all analytically integrable vector fields. However, there are many other vector fields, which behave very nicely
and which admit a local analytic normalization, but which are \emph{not} analytically integrable. For example, take a generalized
Euler vector field $X = \sum_{i=1}^n a_i x_i \dfrac{\partial}{\partial x_i}$, where $a_i > 0$ are positive numbers. According to Poincaré,
any analytic nonlinear perturbation of this vector field will admit
an analytic normalization, even though it does not admit any analytic first integral. Of particular interest is the class of rationally integrable vector fields, which are integrable in the non-Hamiltonian sense (see \cite{Bogoyavlenskij,Zung_Poincare2002}), but with rational instead of analytic first integrals and commuting vector fields:

\begin{define}
\label{define:DarbouxIntegrable} 
A local analytic vector field $X$ in $(\bbK^n,0)$, where
$\bbK = \bbR$ or $\bbC$, with $X(0) = 0$, is called \textit{\textbf{rationally integrable} }if it satisfies the following
properties:

i) There is a natural number $p$, $1 \leq p \leq n$
and $p$ local rational pairwise commuting linearly independent
vector fields  $X_1 = X, X_2, \hdots, X_p$
(the word rational means that
their coefficients are quotients of analytic functions):
$[X_i,X_j] =0\  \forall\ i,j.$

ii) There exist $q = n-p$ functionally independent rational
common first integrals $F_1,\hdots, F_q$ for the rational
vector fields $X_1,\hdots, X_p,$ i.e., $X_i(F_j) =0\  \forall\ i,j$.
\end{define}

In this paper, we extend the results of \cite{Zung_Birkhoff2005,Zung_Poincare2002} to rationally integrable systems:

\begin{thm} \label{thm:main}
Any local analytic vector field $X$ on $(\bbK^n,0)$, where $\bbK = \bbR$ or $\bbC$, with $X(0) = 0$, which is rationally integrable, admits a local analytic Poincaré-Birkhoff
normalization near $0$. 
\end{thm}

The proof of Theorem \ref{thm:main} presented below is based on the toric approach to normalization \cite{Zung_Birkhoff2005,Zung_Poincare2002}, and the so-called \textit{toric conservation principle} 
\cite{Zung_Integrable2016,Zung_AA2018,ZungThien_Stochastic2017,JiangRatiuZung_Normal2024}. In our case, this toric conservation principle becomes a lemma (Lemma \ref{lem:toricconservation}) that says that if a vector field $X$ has a rational first integral $F$ or a rational commuting vector field $Y$, then the associated torus action of $X$ around the origin will also preserve $F$ or $Y$. 

\section{Semi-invariants and toric conservation}
\label{section:toricconservation}

Without loss of generality, from now on we will work in the complex category, i.e., with $\bbK = \bbC$. (Even when a system is real, its associated torus action will generally act in the complex space anyway.) 

Recall that a rational function $F$ in $(\bbC^n,0)$ can be written as
a product $F = \prod_{j=1}^s G_j^{c_j},$ where $G_j$ are irreducible functions, and $c_j$ are integers. If $c_j$ are arbitrary complex numbers instead of integers, then the function $F$ will be called a \textit{\textbf{Darboux function}} instead of a rational function, and it will be multi-valued in general in that case. 

A (formal or analytic) function $G$ is called a (formal or analytic) \textit{\textbf{semi-invariant}} of an analytic vector field $X$ (that vanishes at the origin in  $(\bbC^n,0)$) if
$X(G)$ is divisiblae by $G$, i.e., $X(G) =  \lambda.G$ for some formal or analytic function $\lambda$. It is probably a well-known fact that for any Darboux first integral (and in particular, for any rational first integral) $F = \prod_{j=1}^s G_j^{c_j}$ of $X$, the factors $G_j$ are semi-invariants of $X$, see, e.g.,
Zhang \cite{Zang_Book2017}. For completeness, we will include this fact, together with a proof, in the following lemma about the toric conservation of rational first integrals and commuting vector fields.
A part of this lemma is due to Walcher \cite{Walcher_Poincare2000} and Zhang \cite{Zang_Book2017}.

\begin{lem}
\label{lem:toricconservation}
Let $X = X^s + X^n + \sum_{k \geq 2} X^{(k)}$ be a formal vector field, vanishing at 0, in 
Poincaré-Dulac normal form, i.e. $[X^s,X] = 0$. 

i) (Walcher \cite{Walcher_Poincare2000}) Let
$F = \sum_{k \geq r} F^{(k)}$ with $F^{(r)} \neq 0$ (where $F^{(k)}$
is homogeneous of degree $k$) be a formal semi-invariant of $X$, 
i.e. $X(F) = \lambda. F,$ for some formal power series 
$\lambda = \sum_{k \geq 0} \lambda^{(k)}$. Then there is a formal power series 
$\beta = 1 + h.o.t.$ such that $F^* := \beta F$ satisfies $X(F^*) = \lambda^*. F^*$,
with $X^s(\lambda^*) = 0$ and $X^s(F^*) = \lambda^{(0)}. F^*$.

ii) Moreover, if $X^s = \sum_{i=1}^\tau \gamma_i Z_i$ where $\tau$ is the toric degree of $X$  and $Z_1,\hdots, Z_\tau$ are the generators of its intrinsic associated torus action then for every $i=1,\hdots,\tau$ we also have
$Z_i(F^*) = \lambda^{(0)i}. F^*$ for some number $\lambda^{(0)i}$.

iii) Let $P = \prod_{i=1}^sG_i^{c_i}$ be a formal Darboux-type first integral of $X$, i.e. $X(P) = 0$, where $c_i$ are complex numbers and $G_i$ are irreducible formal functions. Then $P$ is also a formal Darboux-type first integral of $X^s$, i.e., $X^s(P) = 0$.

iv) Moreover, if $P$ is a formal rational first integral, i.e., $c_i \in \mathbb{Z}$ for all $i$, then $P$ is also preserved by the associated torus action of $X$: we  have
$Z_i(P) = 0$ for every $i=1,\hdots,\tau$. 

v) If $Y$ is a formal rational vector field such that $[X,Y] = 0$
then we also have $[Z_i,Y] = 0$ for $i=1,\hdots,\tau$.
\end{lem}

\begin{proof}
i) Since this lemma is very important for us, let us here recall its
proof given by Walcher  \cite{Walcher_Poincare2000}.
The semi-invariance of $F$ with respect to $X$ is equivalent to
\begin{equation}
X^s(F^{(r+j)}) + X^n(F^{(r+j)}) + X^{(2)}(F^{(r+j-1)}) + \hdots + 
X^{(j+1)}(F^{(r)}) = \lambda^{(0)} F^{(r+j)} +  \hdots +
\lambda^{(j)} F^{(r)}
\end{equation}
for all $j \geq 0$.

Note that $X^s(\lambda^{(0)}) = 0$. Now assume that $X^s(\lambda^{(j)}) = 0$ 
for all $j < k$, and let $\tilde{F} = (1 + \beta_k)F$, with $\beta_k$ beging some homogeneous function of degree $k$. Then
\begin{equation}
\tilde{F}  = F^{(r)} + \hdots + F^{(r+k-1)} + (F^{(r+k)} + F^{(r)} \beta_k) + \hdots
\end{equation}
and
\begin{equation}
X(\tilde{F}) = \tilde{\lambda}\tilde{F}
\end{equation} 
with 
\begin{equation}
\tilde{\lambda} = \lambda^{(0)} + \hdots + \lambda^{(k-1)} + (\lambda^{(k)} + X^s(\beta_k)) + \hdots
\end{equation} 
Due to the semi-simplicity of $X^s$, one can choose $\beta_k$ such that 
$X^s(\lambda^{(k)} + X^s(\beta_k)) = 0$. Thus, the first assertion is proved by
induction on $k$: $\beta$ can be constructed in the form of an infinite
product $\prod_{k=1}^\infty (1 + \beta_k)$, where each $\beta_k$ is homogeneous of degree $k$. (Such an infinite product converges in the space of formal  power series). Finally, from $X^s(\lambda^*) = 0$ one deduces that $X^s(F^*) = \lambda^{(0)}. F^*$, again by induction and by the semi-simplicity of $X^s$. 

ii) For any formal power series $H$, the fact that $X^s(H) = \alpha H$
for some number $\alpha$ (here $H = F^*$ and $\alpha = \lambda^{(0)}$)
automatically implies that
$Z_i(H) = \sqrt{-1}\alpha_{i}. H$ for every $i=1,\hdots,\tau$ and for appropriate integers $\alpha_{i}$. This can be seen directly by looking at the Taylor expansion of 
$X^s(H) = \sum_{i=1}^\tau \gamma_i Z_i(H)$ 
term by term in a coordinate system which diagonalizes the $Z_i$. Note that $\sum_{i=1}^\tau \gamma_i 
 \sqrt{-1} \alpha_i = \alpha$,
which determines the values of $\alpha_i$ uniquely from the 
value of $\alpha$, due to the incommensurability of the numbers
$\gamma_1,\hdots, \gamma_\tau$.

iii) The condition that $X(P) = 0$ can be rewritten as
\begin{equation}
\sum_{i=1}^s c_i \dfrac{X(G_i)}{G_i} = 0.
\end{equation}
It then follows from the irreducibility of the $G_i$ that $X(G_i)$ is divisible by $G_i$, i.e. each $G_i$ is a semi-invariant of $X$:
$X(G_i)= \lambda_i.G_i$, where $\lambda_i = \sum_k \lambda_i^{(k)}$
is a formal power series. In particular, we have
\begin{equation}
\sum_{i=1}^s c_i \lambda_i = 0 \quad \text{and} \quad 
\sum_{i=1}^s c_i \lambda_i^{(0)} = 0.
\end{equation}

Using Assertion i) of this lemma, we can multiply each $G_i$ by some invertible formal power series $\beta_i$ such that the new functions $G_i^* = \beta_i G_i$ satisfy $X(G_i^*) = \lambda_i^* G_i^*$ with $X^s(\lambda_i^*) = 0$
and $X^s(G_i^*) = \lambda_i^{(0)}. G_i^*$. The new Darboux-type function
$P^* = \prod_{i=1}^s(G_i^*)^{c_i} = \beta P$ satisfies $X(P^*) = \Lambda^* P^*$ with $X^s(\Lambda^*) = 0$
and $X^s(P^*) = \Lambda^{*(0)}. P^*$. On the other hand, $X(P^*) = X(\beta P)
= X(\beta)P + \beta X(P) = X(\beta) P$, hence $\Lambda^*= X(\beta)$. Therefore the condition $X^s(\Lambda^*) = 0$ means that $X^s(X(\beta))= 0$,
which implies that $X^s(\beta) = 0$ because $X^s$ is the semisimple part of $X$. On the other hand, we have that
\begin{equation}
\dfrac{X^s(P^*)}{P^*} = \sum_i c_i \dfrac{X^s(G_i^*)}{G_i^*} = \sum_i c_i \lambda_i^{(0)} = 0.
\end{equation}
Since both $\beta$ and $P^*$ are preserved by $X^s$, the quotient $P = P^*/\beta$ is also preserved by $X^s$, i.e. we have $X^s(P) = 0$.

iv) According to the previous assertions,
each factor $G_j$ is a semi-invariant of $X$, hence also a semi-invariant of every $Z_i$ ($i =1,\hdots,\tau$), and we can modify the $G_j$ (by multiplying them with appropriate invertible functions,
without modifying the total product $P = \prod_{i=1}^sG_i^{c_i}$) so that they become semi-invariants of the  $Z_i$ with constant multiplicators 
(i.e. they are eigenvectors for the action of $Z_i$ on the space of formal functions):
\begin{equation}
Z_i(G_j) = \sqrt{-1} \lambda_{ij}. G_j
\end{equation}
with $\lambda_{ij} \in  \mathbb{Z}$.
The equality $X^s(P) = 0$ means that
\begin{equation}
\label{eqn:Sum0}
\sum_{i=1^\tau}\gamma_i(\sum_{j=1}^s c_j \lambda_{ij} ) =0,
\end{equation}
which implies that
\begin{equation}
\sum_{j=1}^s c_j \lambda_{ij} =0. 
\end{equation}
for every $i$, because of the incommensurability of the
$\gamma_i$ and the fact that the numbers $c_j$
and $\lambda_{ij}$ are integers when $P$ is a rational function.

v) Write $Y$ as $Y =Y'/F$, where $F$ is an analytic for formal function, and $Y'$ is an analytic or formal vector field which is not divisible by any irreducible factor of $F$. The condition $[X,Y] = 0$ means that
$X(F) Y' = F [X,Y']$, which implies that $X(F)$ is divisible by $F$, i.e.
$F$ is a semi-invariant of $X$. Invoking Assertion ii) of this lemma and multiplying $F$ by an invertible function if necessary, we may assume that
$X(F) = \lambda.F$ such that $Z_i(\lambda) = 0$ and $Z_i(F) = \lambda_i.F$
for all $i = 1,\hdots,\tau$, where $\lambda$ is a function which is invariant with respect to the associated torus action, and $\lambda_1,\hdots, \lambda_\tau$ are constants such that
$\sum_{i=1}^\tau \gamma_i \lambda_i= \gamma(0).$ ($X^s = \sum \gamma_i Z_i$ is the semisimple part of $X$). The equation
$X(F) Y' = F [X,Y']$ implies that $[X,Y'] = \lambda. Y'$. By looking at the Taylor expansion of this equation in a coordinate system which diagonalizes the $Z_i$, degree by degree, one obtains by induction that $[Z_i, Y'] = \lambda_i Y'$ (for every $i=1,\hdots,\tau$). For example, with $Y' = \sum_{i \geq r} Y'^{(i)}$, $X = X^s + X^n + \sum_{i \geq 2} X^{(i)}$ and
$\lambda = \sum_{i \geq 0} \lambda^{(i)}$, at the lowest degree we have
\begin{equation}
[X^s+X^n, Y'^{(r)}] = \lambda^{(0)} Y'^{(r)},
\end{equation}
which implies that $Y'^{(r)}$ must be an eigenvector with respect to the Lie bracket operator $[X^s,.]$ with eigenvalue $\lambda^{(0)}$, and hence it must also be an eigenvector of
the operator $[Z_i,.]$ with eigenvalue $\lambda_i$ for each $i = 1,\hdots, \tau$. At the next degree we have
\begin{equation}
[X^s+X^n, Y'^{(r+1)}] - \lambda^{(0)}Y'^{(r+1)} = \lambda^{(1)}Y'^{(r)} - [X^{(2)},Y'^{(r)}].
\end{equation}
The right hand side of the above equation is an eigenvector of
$[Z_i,.]$ with eigenvalue $\lambda_i$ for each $i = 1,\hdots, \tau$, and so is the left hand side. But, since $Z_i$ commutes
with $X^s+X^n$, it implies that either $Y'^{(r+1)}$ must be an
an eigenvector of $[Z_i,.]$ with eigenvalue $\lambda_i$ or
$[X^s+X^n, Y'^{(r+1)}] - \lambda^{(0)}Y'^{(r+1)} = 0$. But if
$[X^s+X^n, Y'^{(r+1)}] - \lambda^{(0)}Y'^{(r+1)} = 0$ then we can also deduce from this equality that $[Z_i,Y'^{(r+1)}] =
\lambda_i Y'^{(r+1)}$. So in any case we have $[Z_i,Y'^{(r+1)}] =
\lambda_i Y'^{(r+1)}$ (for every $i$). Similarly for the higher degrees.

The equalities  $[Z_i,Y'] = \lambda_i Y'$ and $Z_i(F) = \lambda_i F$ imply that $[Z_i,Y] = [Z_i,Y'/F] = 0$, i.e. $Y$ is preserved by the intrinsic associated torus action. Assertion v) is proved.
\end{proof}

\textbf{Remark.}
In a preprint written some years ago (arXiv:1803.04800, not submitted for publication anywhere) we wrongly thought that the toric conservation principle also works for Darboux-type functions, but it is not true.
In fact, Assertion iv) of the above lemma holds only for rational first integrals, and not for multi-valued Darboux integrals in general, and it is very easy to construct examples of Darboux first integrals that are not conserved by associated torus actions of a diagonal linear vector field. 
This is the reason why we are not able to extend our main result to the case of Darboux-integrable systems, where the first integrals are required to be only Darboux functions, and not rational functions.

\section{Sketch of the proof of Theorem \ref{thm:main}}
\label{section:proof}

In order to explain our proof, we will divide it into several steps.

\textbf{Step 1: Initialization. Fixing of notations.}

Consider a rational integrable analytic vector field $X$ 
on $(\mathbb{C}^n,0)$ with $X(0) = 0$. 
Denote by $F_i = \prod_{j=1}^{s_i} G_{ij}^{c_{ij}}$ ($i=1,\hdots,q$)
the common rational first integrals of the commuting rational vector fields $X_1=X,X_2,\hdots,X_p$ ($p+q=n$) given by the rational integrability condition of $X$. Here $G_{ij}$ are irreducible local analytic functions and
$c_{ij}$ are integers.
For each $i = 2,\hdots,p$, denote by $H_i$ a nontrivial local 
analytic function such that $H_iX_i$ is a local analytic vector field.
Denote by $\mathcal{S}$ the union of the local analytic sets
$ \{G_{ij} (x) = 0 \}$,$ \{H_{i} (x) = 0 \}$,
$ \{ dF_1 \wedge \hdots \wedge dF_q(x) = 0 \}$ and
$\{X_1 \wedge \hdots \wedge X_p (x) = 0\}$ in $(\mathbb{C}^n,0)$, i,e, the local set of all possible singular points of our system. A priori, $\mathcal{S}$
is an analytic subset of $(\mathbb{C}^n,0)$ of codimension at least 1,
and we will call it the \textbf{\textit{singularity set}} of the system.

As before, denote by $\tau$ the toric degree of $X$, $X = X^S + X^N$ the intrinsic
(a-priori formal) decomposition of $X$ into the sum of its semisimple part and its nilpotnent part, and $X^S = \sum_{i=1}^\tau \gamma_i Z_{i}$ into a linear combination of the generators $Z_1,\hdots,Z_\tau$ of its associated torus action, which is a-priori only formal, but we want to show that this torus action is in fact analytic.  

Denote by $d(z,\mathcal{S})$ the distance from a point $z$ in 
$(\mathbb{C}^n,0)$ to $\mathcal{S}$ (with respect to a given norm
on $(\mathbb{C}^n,0)$, doesn't really matter which norm). By 
\L ojasiewicz-type inequalities (see, e.g., \cite{Lojasiewicz}),
there exist a natural
number $D$ and a positive number $\delta >0$
such that for any element 
$A$ of a finite family of holomorphic and rational functions and tensor fields
on $(\mathbb{C}^n,0)$ that we will use in this paper (our rational vector fields
$X_i$ and the log-derivatives $X_i(F_j)/F_j$ of our rational first
integrals belong to this family), which have all of their zeros and poles in
$\mathcal{S}$, we have
\begin{equation}
\label{eqn:Lojasiewicz}
(d(z, \mathcal{S}))^{-D} \geq \|A (z)\| \geq (d(z, \mathcal{S}))^D
\end{equation}
for any $z \in (\mathbb{C}^n,0)$ with $\|z\| < \delta$. These \L ojasiewicz-type inequalities will be important for our construction.

\textbf{Step 2: Construction of a system-preserving torus action in a union of open domains outside of $\mathcal{S}$}

For each natural number $m$, we will construct an open subset $\mathcal{U}_m$
near the origin of $(\mathbb{C}^n,0)$ outside of $\mathcal{S}$ (i.e., $\mathcal{U}_m \cap \mathcal{S} = \emptyset$), together with an effective holomorphic torus $\mathbb{T}^\tau$-action in $\mathcal{U}_m$, which satisfies the following properties:

1a) There is a small open neighborhood $\mathcal{V}_m$ of the origin in $(\mathbb{C}^n,0)$ ($\mathcal{V}_n$ gets smaller when $m$ gets bigger: $\mathcal{V}_{m+1} \subset \mathcal{V}_m$ for every $m$)
such that the complement $\mathcal{V}_m \setminus \mathcal{U}_m$ is a \emph{horn of sharpness order at least $m$} around $\mathcal{S}$ in the sense that 
\begin{equation}
\label{eqn:hornordern}
d(z, \mathcal{S}) \leq \|z\|^m
\end{equation}
for any point $z \in \mathcal{V}_m \setminus \mathcal{U}_m$. 

1b) The about torus $\mathbb{T}^\tau$-action preserves our system, i.e., it preserves all the first integrals $F_1,\hdots,F_q$ and commuting vector fields $X_1=X,X_2,\hdots,X_p$ of the system. 

1c) For any two natural numbers $k$ and $m$, the constructed torus actions  
on $\mathcal{U}_k$ and on $\mathcal{U}_m$ will coincide on their intersection
$\mathcal{U}_k \cup \mathcal{U}_m$. In other words, we have a holomorphic torus 
$\mathbb{T}^\tau$-action on the union $\mathcal{U} = \bigcup_{m=1}^\infty \mathcal{U}_m$ which preserves the system.

The construction of the above torus action (for each $m \in \mathbb{N}$) will be done via a geometric approximation method, by the following substeps:

1d) Construct the associated torus action for an analytically normalizable system which approximates our system up to a sufficiently high order (such a system always exists, by a truncated normalization process à la Poincaré-Birkhoff), and then that this torus action constructed in 1d) ``almost conserves'' our system in $\mathcal{U}_n$, in a sense to be made more precise in Lemma 
\ref{lemma:orderM}. 

1e) Due to this almost conservation, we can project the generating vector fields of the torus action constructed in 1d) to the level 
sets of the system (i.e. the subsets where all the first integrals are constant), and then use the fact that these level sets have a natural affine structure generated by our original commuting vector fields to show that these projected vector fields can be approximated (in a unique way) by other vector fields on our level sets which generate a $\mathbb{T}^\tau$-action which preserves the system.   

Details of the above construction will be given in Sections
\ref{section:almostconservation} and \ref{section:construction}.

\textbf{Step 3. Analytic extension and conclusion.}

By construction, 
the complement $\mathcal{H} = \mathcal{V}_1 \setminus \mathcal{U}$ of the union
$\mathcal{U} = \bigcup_{n=1}^\infty \mathcal{U}_n$ is a \textbf{\textit{sharp horn}} around the singularity set $\mathcal{S}$ in the sense that the degree of sharpness tends to infinity when it tends to the origin. In other words, we have
\begin{equation}
\label{eqn:sharphorn}
\lim_{\varepsilon \to 0} \inf_{z \in \mathcal{H}, \|z\| = \varepsilon}
\left| \dfrac{\ln d(z,\mathcal{S})}{\ln \varepsilon} \right| = \infty .
\end{equation}
Hence, we can  invoke the following key holomorphic extension lemma from \cite{Zung_Birkhoff2005}:
\begin{lem}[\cite{Zung_Birkhoff2005}]
\label{lem:extension}
Let $\mathcal{S}$ be a complex analytic subset of positive codimension in  $({\mathbb C}^n,0)$ and $\mathcal{U}$ be a subset of $({\mathbb C}^n,0)$ whose complement  $\mathcal{H} = ({\mathbb C}^n,0) \setminus \mathcal{U}$ is a sharp-horn of  $\mathcal{S}$.
Then any bounded holomorphic function on $\mathcal{U}$
admits a holomorphic extension in a neighborhood of $0$ in ${\mathbb C}^n$.
\end{lem}

Using the above lemma, we get an analytic extension of the generating vector fields (via the analytic extension of their coefficients in a coordinate system)
of the torus $\mathbb{T}^\tau$-action from $\mathcal{U}$ to an open neighborhood of the origin in $({\mathbb C}^n,0)$. By analytic continuation, these analytically extended vector fields still generate a holomorphic torus action which preserve the system. By dimension consideration, this torus action cannot be anything else but the intrinsic torus action associated to our system. 
Thus, we have proved that the associated torus action of our system is analytic, and therefore the system is analytically normalizable.

\section{Torus actions that almost conserve the system}
\label{section:almostconservation}

By the classical method of step-by-step normalization,
there is an infinite sequence of local analytic coordinates systems
$(x_{1,m},\hdots, x_{n,m})$  ($m \in \mathbb{N}$) on $(\mathbb{C}^n,0)$ with the following properties:

i) $x_{i,m} (z) - x_{i,m'} (z) = o(\|z\|^{\min(m,m')})$ on $(\mathbb{C}^n,0)$ for any $m,m', i$, where $\|.\|$ denotes a norm on $(\mathbb{C}^n,0)$ (it doesn't matter which norm). 
The formal limit of these coordinate systems
when $m$ goes to infinity is a normal normalization of $X$.

ii) The Taylor expansion of the vector field $X$ in the coordinate system $(x_{1,m},\hdots, x_{n,m})$ is
\begin{equation}
X = X^{s}_m + X^{n}_m + X^{(2)}_m + X^{(3)}_m + \hdots,
\end{equation}
such that 
$[X^{s}_m , X] (z) = o(\|z\|^m)$, i.e. it does not contain terms of order less than or equal to $m$. Note that when we change the coordinate system, we will still keep the vector fields $X^{(k)}_m$ in the above decomposition, though they will no longer be homogeneous in new coordinate systems. 

iii) $X^{(s)}_m $ tends to the semisimple part $X^S$ of $X$ in the formal category when $m$ tends to infinity.

iv) For each $m \in \mathbb{N}$ we have
\begin{equation}
X^{(s)}_m = \sum_{i=1}^\tau \gamma_i Z_{i,m}
\end{equation}
where $\tau$ is the toric degree of $X$ at $0$, each vector field $Z_{i,m}$
is a diagonal linear vector field in the coordinate system 
$(x_{1,m},\hdots, x_{n,m})$,
\begin{equation}
Z_{i,m} = \sqrt{-1} \sum_{j=1}^n \rho_{ij} x_{j,m} 
\dfrac{\partial}{\partial x_{j,m}},
\end{equation}
Here $\gamma_1,\hdots,\gamma_\tau$ are incommensurable complex numbers which
do not depend on $m$, and the numbers $\rho_{ij} \in \mathbb{Z}$ are integers which do not depend on $m$.

v) For each $m \in \mathbb{N}$, the vector fields 
$Z_{1,m},\hdots,Z_{\tau,m}$ are the generators of a local effective torus $\mathbb{T}^\tau$-action on $(\mathbb{C}^n,0)$ which preserves 
$X$ up to order $m$.

vi) $Z_{i,m}$ tends to $Z_i$ ($i=1,\hdots,\tau$) formally when $m$ tends
to infinity. More precisely, the order of $Z_i - Z_{i,m}$ (i.e. the lowest degree of its non-zero terms) is at least $m+1$ for every $m$. Recall that  $Z_{1},\hdots,Z_{\tau}$ denote
the generators of the intrinsic associated formal torus $\mathbb{T}^\tau$-action
and $\sum_{i=1}^\tau \gamma_i Z_{i} = X^S$ is the semisimple part of $X$. 

The following lemma says that the analytic torus actions generated by $(Z_{1,m},
\hdots, Z_{1,m})$ will almost conserve our system (not only the vector field $X$, but also the first integrals and the other commuting vector fields) up to any given order, provided that $m$ is large enough for the required order of conservation:

\begin{lem}
\label{lemma:orderM} 
For any increasing map $\nu: \mathbb{N} \to \mathbb{N}$, there exists  an
increasing map 
$\mu: \mathbb{N} \to \mathbb{N}$ and a decreasing infinite sequence of positive numbers $(\varepsilon_m)_{m \in \mathbb{N}}$, such that the following inequalities about the almost conservation with respect to the vector fields $Z_{i,m}$ are satisfied: For any $i= 1,\hdots, \tau$, any $j=1,\hdots,p$, any $m \in \mathbb{N}$, any
$\mu \geq \mu(m)$, any $k=1,\hdots,q$,
and any point $z \in \mathbb{C}^n $ such that
$\|z\| < \varepsilon_\mu$  and $d(z,\mathcal{S}) > \|z\|^m$ we have:
\begin{equation}
\label{eqn:NearZero1}
\|[Z_{i,\mu},X_j] (z)\| \leq \|z\|^{\nu(m)}
\end{equation}
and
\begin{equation}
\label{eqn:NearZero2}
\|Z_{i,\mu} (F_k) (z)\| \leq \|z\|^{\nu(m)}
\end{equation}
\end{lem}

\begin{proof}
We will prove Inequality \eqref{eqn:NearZero1} (for a given couple of indices
$i,j$). The proof of Inequality
\ref{eqn:NearZero2} is absolutely similar.

Write $X_j$ as $X_j = X_j'/H_j$, where $X_j'$ and $H_j$ are local analytic and
$X_j'$ is not divisible by any irreducible factor of $H_j$. According to Lemma 
\ref{lem:toricconservation}, we have $[Z_i, X_j] = 0$, which means that 
\begin{equation}
Z_i(H_j) X_j' - H_j [Z_i, X_j'] = 0.
\end{equation}
Since the order of $Z_{i,\mu} - Z_i$ is at least $\mu + 1$ for every $\mu$, we have
that the order of $Z_{i,\mu}(H_j) X_j' - H_j [Z_{i,\mu}, X_j'] =
(Z_{i,\mu}-Z_i) (H_j) X_j' - H_j [Z_{i,\mu} - Z_i, X_j']$ 
is at least $\mu-1$, and so we may assume that
\begin{equation}
\|Z_{i,\mu}(H_j) X_j' (z) - H_j [Z_{i,\mu}, X_j'] (z) \| \leq \|z\|^{\mu-2}
\end{equation}
for every $\mu$ and every $\|z\| < \varepsilon_\mu$ (provided that 
$\varepsilon_\mu$ is sufficiently small), which implies that
\begin{equation}
\| [Z_{i,\mu}, X_j] (z) \| \leq \dfrac{\|z\|^{\mu-2}}{|H_j(z)|^2}.
\end{equation}
Applying the \L ojasiewicz inequality \eqref{eqn:Lojasiewicz} and the inequality
$d(z,\mathcal{S}) \geq \|z\|^m$ , we get
\begin{equation}
\| [Z_{i,\mu}, X_j] (z) \| \leq \dfrac{\|z\|^{\mu-2}}{d(z,\mathcal{S})^{2D}}
\leq \|z\|^{\mu-2 - 2Dm} ,
\end{equation}
so it suffices to take $\mu(m) = \nu(m) + 2 + 2Dm$ for the inequality \eqref{eqn:NearZero1} to be satisfied. 
\end{proof}

\section{Construction of torus actions that conserve the system}
\label{section:construction}

The construction given below is similar to the one given in
\cite{Zung_Birkhoff2005,Zung_Poincare2002}.

 For every natural number $m$, we will construct (inductively on $m$, with backward compatibility), 
 by approximation, a system-preserving 
 ${\mathbb T}^\tau$ action on an open subset
 $\mathcal{U}_m$ of $({\mathbb C}^n,0)$ whose complement is a horn of sharpness order at least $m$ around $\mathcal{S}$.

Take a small neighborhood $\mathcal{V}_m$ of the origin, and the horn 
$\mathcal{H}_m = \{z \in \mathcal{V}_m, d(z,\mathcal{S}) \leq \|z\|^m\}$ of sharpness order $m$ around $\mathcal{S}$
in $\mathcal{V}_m$. We can choose $\mathcal{V}_m$ sufficiently small so that the inequalities in Lemma \ref{lemma:orderM} hold, where $\nu(m)$ is a sufficiently large number so that our approximation arguments below will hold. 

Consider the torus $\mathbb{T}^\tau$ action generated by $(Z_{1,\mu}, \hdots, Z_{\tau,\mu})$, where the number 
$\mu = \mu(m)$ is also from Lemma \ref{lemma:orderM}. This torus action almost preserves the first integrals and the commuting vector fields, up to the order $\nu(m)$, according to Lemma \ref{lemma:orderM}. 
Denote by $\mathcal{O}_\mu(z)$ the orbit of this torus action through $z$. Note that these 
$\mathbb{T}^\tau$ actions (for different $m$) deform the norm of $(\mathbb{C}^m,0)$ but not by too much and in a uniform way, i.e.,
there exists a positive constant $C$ such that
$C^{-1}\|z\| \leq \|z'\| \leq C\|z\|$ for any $z$ such that $\|z\|$ is smaller than a certain positive number depending on $m$ and for any  $z' \in \mathcal{O}_{\mu(m)}(z)$.

Consider an arbitrary point $z \in \mathcal{V}_m \setminus \mathcal{H}_m$. By definition, $\|z\| < \varepsilon_{\mu(m)}$ (we can choose $\varepsilon_{\mu(m)}$ as small as need be by shrinking $\mathcal{V}_m$)
but $d(z,\mathcal{S}) > \|z\|^m$, so the inequalities
in Lemma \ref{lemma:orderM} apply to $z$, with $\mu = \mu(m)$. The  \L ojasiewicz inequality \eqref{eqn:Lojasiewicz} implies that
\begin{equation}
\label{eqn:Lojasiewicz}
\|z\|^{-Dm} \geq (d(z, \mathcal{S}))^{-Dm} \geq \|A (z)\| \geq (d(z, \mathcal{S}))^{D} \geq \|z\|^{Dm}
\end{equation}
(where $D$ is a positive constant which does not depend on $m$)
for every $z \in \mathcal{V}_m \setminus \mathcal{H}_m$, and $A$ is any element from our finite family of non-trivial analytic and rational functions and tensors in our study (that do not have zeros or poles outside of $\mathcal{S}$). Note that this finite family includes the first integrals and their irreducible factors, the coefficients of the commuting vector fields, the (coefficients of) the wedge products of the differentials of the first integrals and of the commuting vector fields. It depends only on the system itself, and not on $m$, and in particular, things like $[Z_{i,\mu},X_j] (z)$ are not in this family.

We can choose $\nu(m) > Dm$ sufficiently large, so that the norms of $[Z_{i,\mu},X_j](z)$ 
$Z_{i,\mu}(F_j)(z)$ are much smaller than the norms of the values at $z$ of the objects from our fixed finite family 
for every $z \in \mathcal{V}_m \setminus \mathcal{H}_m$. In particular, the variation of the first integrals 
on $\mathcal{O}_\mu(z)$ is very small (of order $\|z\|^{\nu(m)}$) compared to the magnitude at $z$ of the objects from our fixed finite family. It implies that the orbit $\mathcal{O}_\mu(z)$ stays ``sufficiently away'' from 
$\mathcal{S}$, in the sense that there is a number $m' \geq m$ (which may depend on the system but will not 
depend on our choice of $\mu(m)$, $\nu(m)$, etc.) such that $d(z',\mathcal{S}) > \|z'\|^{m'}$ for any
$z' \in \mathcal{O}_\mu(z)$.

The orbit $\mathcal{O}_\mu(z)$ is almost tangent to the level set $\mathcal{L}_z = \{z' \in (\mathbb{C}^n,0), F_i(z')= F_i(z) \; \forall \; i=1,\hdots,q\}$, and we can project $\mathcal{O}_\mu(z)$, together with the vector field $Z_{1,\mu}, \hdots, Z_{\tau,\mu}$ on it, orthogonally to $\mathcal{L}_z$, to obtain a ``projected torus''
$\tilde{\mathcal{O}}_\mu(z)$ on $\mathcal{L}_z$ together with ``projected'' tangent vector fields
$\tilde{Z}_{1,\mu}, \hdots, \tilde{Z}_{\tau,\mu}$ on $\tilde{\mathcal{O}}_\mu(z)$. This projected torus 
$\tilde{\mathcal{O}}_\mu(z)$ is also sufficiently far away from $\mathcal{S}$, in the sense that 
$d(y,\mathcal{S}) > \|y\|^{m'}$ for any $y \in \tilde{\mathcal{O}}_\mu(z)$. 

In a ``sufficiently large'' neighborhood of $\tilde{\mathcal{O}}_\mu(z)$ in $\mathcal{L}_z$ we have a natural regular flat affine structure defined by the commuting vector fields $X_1,\hdots, X_q$ (they are tangent to every level set, and they are linearly independent on  $\mathcal{L}_z$ which is of dimension $q$). The vector fields $\tilde{Z}_{1,\mu}, \hdots, \tilde{Z}_{\tau,\mu}$ are almost constant in this flat affine structure, in the sense that if we write $\tilde{Z}_{i,\mu} (y) = \sum_{j=1}^q b_{ij} (y)  X_j$ on $\tilde{\mathcal{O}}_\mu(z)$ 
then the functions $b_{ij} (y)$ are almost constant (their variations are very small compared to their values). 
It follows, by integrating on $\mathcal{L}_z$, that there exist unique numbers
$\hat{b}_{ij}$, that do not depend on  $y \in \mathcal{L}_z$ and are very close to the numbers $b_{ij} (z)$, 
such that the time-1 flows of the vector fields $\sum_{j=1}^q \hat{b}_{ij} X_j$ on $\mathcal{L}_z$ fixes the point 
$z$ (they send $z$ to $z$). By the flat nature of the affine structure on $\mathcal{L}_z$, it implies that the vector fields  generate a $\mathbb{T}^\tau$ action in a neighborhood of $\tilde{\mathcal{O}}_\mu(z)$ in 
$\mathcal{L}_z$ that preserves the system (i.e. preserves our level set and the commuting vector fields $X_1,\hdots, X_p$). Denote the by $\mathcal{O}(z)$ the orbit through $z$ of this $\mathbb{T}^\tau$  action. By uniqueness and continuity arguments, when we very $z \in \mathcal{V}_m \setminus \mathcal{H}_m$, these torus actions (one constructed for each $z$ will fit together to form a
$\mathbb{T}^\tau$ action on the union $\mathcal{U}_m = \bigcup_{z \in \mathcal{V}_m \setminus \mathcal{H}_m} \mathcal{O}(z)$. Note that $\mathcal{U}_m$ contains $\mathcal{V}_m \setminus \mathcal{H}_m$, and so its complement is
a horn of sharpness order at least $m$ around $\mathcal{S}$.

Again, due to the uniqueness of the above numbers $\hat{b}_{ij}$ (which depend on the point $z$ but not on the othe choices), if $z \in U_m \cap U_{m'}$ for two different natural numbers $m$ and $m'$ then the generators of the torus action on $U_m$ will coincide with the generators of the torus action on $U_{m'}$ at $z$. In other words, we get a
system-preserving torus $\mathbb{T}^\tau$-action on the union 
$\mathcal{U} = \bigcup_{m \in \mathbb{N}} \mathcal{U}_m$. The fact that the coefficients of the generators of this 
$\mathbb{T}^\tau$-action are bounded functions on $\mathcal{U}$ are bounded functions are also clear by the construction. Since the complement of $\mathcal{U}$ is a sharp horn around $\mathcal{Z}$ by construction, this torus action admits an analytic extension to a neighborhood of the origin in $(\mathcal{C}^n,0)$ by Lemma \ref{lem:extension}. Thus, we have get a local analytic torus $\mathbb{T}^\tau$ action which preserves our system. 

By construction, the linear part of this analytic torus action coincides with the linear part of the (a priori formal) associated torus action of the vector field $X$. By uniqueness, the associated torus action of $X$ cannot be anything else than this analytic torus action. Thus, we have proved that the associated torus action of $X$ is analytic, and hence $X$ admits an analytic normalization. 

\section{Final remarks and acknowledgements}

In view of Theorem \ref{thm:main}, we have the following conjecture: the main theorem of Morales--Ramis--Simo on Galoisian obstructions to integrability 
of meromorphic Hamiltonian systems \cite{MRS_Galois2007}, and its
extension to the non-Hamiltonian case \cite{AyoulZung_Galois2010}, are also 
valid for the case of rationally integrable systems. 

Another related interesting open question, posed to me by my colleague Emmanuel Paul, is: can Theorem \ref{thm:main} be used to give a geometric proof for analytic normalization results involving Diophantine conditions à la Siegel--Bruno? For resonant analytic vector fields, Bruno's theorem \cite{Bruno_Book1989} says that a Diophantine condition together with the so-called Bruno's \emph{Condition A} imply the existence of an analytic normalization. This Condition A (which says that there is a formal normal form of $X$ of the type $f.X^{(1)}$ where $X^{(1)}$ is linear), is a highly nontrivial formal condition, which is somehow related to the existence of first integrals (it reduces the problem of first integrals of $X$ to the problem of first integrals of the linear part of $X$), and it begs the following question: do formal rational integrability plus Diophantine conditions for analytic vector fields imply analytic rational integrability (and the existence of analytic normalization)?

What about normalization of Darboux-integrable systems, where first integrals can be multi-valued Darboux functions instead of rational functions. I have no idea at this moment. As we mentioned in the paper, the toric conservation principle does not apply to Darboux-type functions, so our toric approach does not work for them. 

This paper was inspired by a recent book by Xiang Zhang \cite{Zang_Book2017}
on the integrability and normalization of dynamical systems, which contains many results on Darboux integrability. I would like to thank Xiang Zhang for giving me a copy of this very nice book and for many useful discussions. I'm also very thankful to Tudor Ratiu and Kai Jiang for inviting me to Shanghai and Beijing many times, and for collaborating with me on the toric conservation principle
\cite{JiangRatiuZung_Normal2024}.

The main result of this paper was briefly announced (erroneously, because at that time I talked about normalization of Darboux-integrable systems) in my talks
in Nice and Toulouse in 2018. I would like to thank the participants of
these talks, in particular Emmanuel Paul, Jean-Pierre Ramis and Laurent Stolovitch, for interesting discussions on the subject. 
A first draft of this paper (which contained the embarrassing mistake about the toric conservation of Darboux integrals) was written  during my visit to the Center for Geometry and Physics, Institute for Basic Science, POSTECH, South Korea. I would like to thank  this Center, and its Director Yong-Geun Oh in particular, for the invitation, warm hospitality and excellent working conditions. 

Due to my non-academic activities as the founder and director of a software company called Torus Actions since 2019, I was not able to revise this first draft until very recently, when I was invited by Alexander Ivanov to submit a paper on the occasion of T.A. Fomenko's 80th birthday. I would like to thank Alexander and my Russian friends and colleagues, especially Alexey Bolsinov and Andrei Oshemkov, for this invitation and for many discussions on integrable systems during many years.

My first encounter with torus actions in dynamical systems happened in 1988-1989, 
when I became a third-year under-graduate student under the supervision of A.T. Fomenko. Anatoly Timoveevich made me read his books and articles on integrable Hamiltonian systems (e.g., \cite{Fomenko_Morse1986,Fomenko_Integrability1988, FomenkoZieschang_Integrable1990}), and gave me the following problem: show that an integrable system with 2 degrees of freedom near a nondegenerate hyperbolic singular level of rank 1 can be perturbed in such a way that it remains integrable but all the hyperbolic singularities become simple (i.e., each singular level contains only one circle of hyperbolic singular points). This result would be similar to a classical result in Morse theory, where a function can be perturbed so that each singular level set will contain only one singular point. I eventually solved this problem some months later by discovering a hidden torus action which preserves the system, and this result became my first research paper, published in 1990 \cite{Zung_Bott1990}. Since then, the tori became inseparable from me, and throughout my research career I wrote many papers about them. I feel very lucky to be a student of Atatoly Timofeevich and to follow his passion for mathematics in general and for integrable systems in particular, and it is a great pleasure for me to dedicate this paper to him on the occasion of his 80th birthday, with gratitude. 

\bibliographystyle{amsplain}

\end{document}